\documentclass{amsart}
\usepackage{amssymb, mathrsfs, bbm}
\usepackage[T1]{fontenc}
\usepackage[sc, osf]{mathpazo}
\selectfont
\usepackage[protrusion=true, expansion=true]{microtype}

\usepackage[all, arc]{xy}
\SelectTips{eu}{}
\entrymodifiers={+!!<0pt,\fontdimen22\textfont2>}

\usepackage[pdftex, colorlinks=true, linkcolor=blue, citecolor=magenta, linktocpage]{hyperref}

\theoremstyle{plain}
\newtheorem{thm}[subsubsection]{Theorem}
\newtheorem{lemma}[subsubsection]{Lemma}
\newtheorem{prop}[subsubsection]{Proposition}
\newtheorem{cor}[subsubsection]{Corollary}

\theoremstyle{definition}

\theoremstyle{remark}

\theoremstyle{definition}

\numberwithin{equation}{subsubsection}

\def\cA{\mathcal{A}}

\def\cE{\mathcal{E}}

\def\cL{\mathcal{L}}

\def\cO{\mathcal{O}}

\def\11{\mathbf{1}}

\def\CC{\mathbf{C}}

\def\HH{\mathbf{H}}

\def\MM{\mathbf{M}}

\def\PP{\mathbf{P}} 
 
\def\RR{\mathbf{R}}

\def\ZZ{\mathbf{Z}}

\def\fn{\mathfrak{n}}

\def\C{\mathbf{C}}

\def\conv{\star}
\def\dim{\mathrm{dim}} 

\def\End{\mathrm{End}}

\def\Ext{\mathrm{Ext}}

\def\Hom{\mathrm{Hom}}

\def\Perv{\mathrm{Perv}}

\def\pH{\vphantom{H}^p\!H}

\def\T{\mathbf{T}}

\newcommand{\mapright}[1]{\xrightarrow{#1}}

\newcommand{\const}[1]{\underline{#1}}
\newcommand{\dgDer}[1]{\mathrm{dgDer-}{#1}}

\newcommand{\qtimes}[1]{\times^{#1}}
\newcommand{\ttimes}{\mathop{\widetilde{\boxtimes}}}

\title{Some geometric facets of the Langlands correspondence for real groups}
\author{R. Virk}
\begin{document}
\maketitle
\renewcommand{\thesubsection}{\arabic{subsection}}
%
%
%
\subsection{Introduction}
Let $G^{\vee}$ be a reductive algebraic group defined over $\RR$. The local Langlands correspondence \cite{L} describes the set of equivalence classes of irreducible admissible representations of $G^{\vee}(\RR)$ in terms of the Weil-Deligne group and the complex dual group $G$. Roughly, it partitions the set of irreducible representations into finite sets (L-packets) and then describes each L-packet. Since Langlands' original work, this has been refined in several directions. Most relevant to this document is the work of Adams-Barbasch-Vogan \cite{ABV}.

The key construction in \cite{ABV} is that of a variety (the parameter space) on which $G$ acts with finitely many orbits; each L-packet is re-interpreted as an orbit, and representations in the L-packet as the equivariant local systems supported on it. Adams-Barbasch-Vogan further demonstrate that the parameter space encodes significant character level information.

In \cite{So} W. Soergel has outlined a conjectural relationship between the geometric and representation theoretic categories appearing in \cite{ABV}. This relationship, roughly a type of Koszul duality, yields a conceptual explanation for the phenomena observed in \cite{ABV}.\footnote{
The localization theorem of Beilinson-Bernstein also establishes a relationship between representation theory and geometry. However, Soergel's approach is very different: localization leads to geometry on the group itself; Soergel's approach results in geometry on the dual group. 
}
The current note was born in an attempt to settle Conjectures 4.2.2, 4.2.3, 4.2.6 in \cite{So}, and Soergel's `Equivariant Formality' conjecture (implicit in \cite{So}; see \S\ref{s:formality} and \cite[\S0.2]{Lun}). 
These conjectures describe the structure of the geometric categories appearing in \cite{ABV};
the current document has little to say about Soergel's Basic Conjecture (relating the geometric categories to representation theory).

We `\emph{almost}' succeed (see \S\ref{s:formality}).
Soergel's conjectures are formulated in graded versions of our categories (in the sense of \cite[\S4]{So}). These `graded representation theories' are not constructed here because of a rather frustrating reason: we use the language of Hodge modules, and the category of Hodge modules is too large for the purposes of graded representation theory (see \S\ref{s:formality} and footnote \ref{note}).
The main result is Theorem \ref{main} describing the Hodge structure on equivariant $\Ext^{\bullet}$ between simple perverse sheaves on the parameter space.\footnote{
The parameter space is never explicitly mentioned. The translation between the symmetric varieties $G/K$ appearing below and the parameter space is provided by \cite[Proposition 6.24]{ABV}.
The parameter space is essentially a disjoint union of varieties of the form $G/K$.
} 
It yields a host of ancillary results which are of independent interest: Corollary \ref{extparity} and \ref{icparity} (`parity vanishing'); Theorem \ref{mqstable} and Corollary \ref{positivity} (`positivity' of a Hecke algebra module).
An informal discussion regarding Soergel's graded categories is contained in \S\ref{s:formality}.


\textbf{Acknowledgments:}
This work was conceived while I was visiting the Albert-Ludwigs-Universit\"at, Freiburg, in the summer of 2012. I am indebted to W. Soergel for his hospitality, his continuing explanations and patience.\footnote{\label{note}
In January, 2010, W. Soergel explained to me: \emph{``In a way, there should be a better category than what we work with, sort of much more motivic, where these problems disappear. Think about Grothendieck's conjecture: the action of Frobenius on the \`etale
cohomology of a smooth projective variety should be semisimple! So this
non-semisimplicity is sort of due to the fact we are not working with motives,
but with some rather bad approximation, I suggest."} At my glacial pace it has taken me four years to appreciate this (see \S\ref{s:formality}).
}
I am further grateful to W. Soergel and M. Wendt for sharing their beautiful ideas on the use of motivic sheaves in representation theory.

\subsection{Conventions}\label{s:conventions}
Throughout, `variety' = `separated reduced scheme of finite type over $\mathrm{Spec}(\CC)$'.
A \emph{fibration} will mean a morphism of varieties which is locally trivial (on the base) in the \`etale sense. 
Constructible sheaves, cohomology, etc. will always be with $\RR$ or $\CC$ coefficients, and with respect to the complex analytic site associated to a variety.

Given an algebraic group $G$, we write $G^0$ for its identity component. 
Suppose $G$ acts on $X$. Then we write $G_x$ for the isotropy group of a point $x\in X$. Given a principle $G$-fibration $E\to B$, we write $E\qtimes{G}X \to B$ for the associated fibration.\footnote{
Generally, $E\qtimes{G} X$ is only an algebraic space. It is a variety if, for instance, $X$ is quasi-projective with linearized $G$-action; or $G$ is connected and $X$ can be equivariantly embedded in a normal variety (Sumihiro's Theorem). 
One of these assumptions will always be satisfied below.
}

We write $D_G(X)$ for the $G$-equivariant derived category (in the sense of \cite{BL}), and $\Perv_G(X)\subseteq D_G(X)$ for the abelian subcategory of equivariant perverse sheaves on $X$. Perverse cohomology is denoted by $\pH^*$.
Change of group functors (restriction, induction equivalence, quotient equivalence, etc.) will often be omitted from the notation. All functors between derived categories will be tacitly derived.
Both the functor of $G$-equivariant cohomology as well as the $G$-equivariant cohomology ring of a point will be denoted by $H^*_G$.
\subsection{$\mathbf{B\backslash G/K}$}\label{s:bgk}
Let $G$ be a connected reductive group, $\theta\colon G\to G$ a non-trivial algebraic involution, $T$ a $\theta$-stable maximal torus, and $B\supseteq T$ a $\theta$-stable Borel containing it (such a pair $(B,T)$ always exists, see \cite[\S7]{St}).
Write $W$ for the Weyl group. 
Let $K = G^{\theta}$ denote the fixed point subgroup.
Then
\begin{enumerate}
\item $K$ is reductive (but not necessarily connected, see \cite[\S1]{V});
\item $|B\backslash G/K| < \infty$ (a convenient reference is \cite[\S6]{MS});
\item $K$-orbits in $G/B$ are affinely embedded (see \cite[Ch. H, Proposition 1]{M});
\item for each $x\in G/B$, the component group $K_x/K_x^0$ has exponent $2$ \cite[Proposition 7]{V}.
\end{enumerate}

Our primary concern is the category $D_{B\times K}(G)$, for the $B\times K$-action given by
$(b,k)\cdot g = bgk^{-1}$. 
The evident identification of $B\times K$-orbits in $G$, with $B$-orbits in $G/K$, and with $K$-orbits in $G/B$, respects closure relations.
There are corresponding identifications:
$D_B(G/K) = D_{B\times K}(G) = D_K(G/B)$.
These identifications will be used without further comment.

Let $s\in W$ be a simple reflection, $P\supseteq B$ the corresponding minimal parabolic,
and $v$ a $B$-orbit in $G/K$. Then the subvariety $P\cdot v \subseteq G/K$ contains a unique open dense $B$-orbit $s\star v$. Let $\leq$ denote the closure order on orbits, i.e.,
$v\leq w$ if and only if $v$ is contained in the closure $\overline{w}$.
\begin{thm}[{\cite[Theorem 4.6]{RS}}]\label{order}
If $w\in B\backslash G/K$ is not closed, then there exists a simple reflection $s\in W$, and $v\in B\backslash G/K$ such that $v\lneq w$ and $s\star v= w$.
\end{thm}

Let $\pi\colon G/B \to G/P$ be the evident projection. 
Let $x\in G/B$. Set $y=\pi(x)$, and $L_x^s = \pi^{-1}(y)$. Note: $L_x^s\simeq \PP^1$.
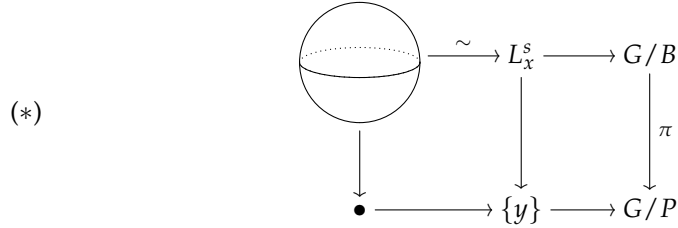
\begin{equation}\label{defpi}\tag{$\ast$}\begin{gathered}
\xymatrix{
\xy (0,0)*\xycircle(8,8){-};
(0,0)*\ellipse(8,2){.};
(0,0)*\ellipse(8,2)__,=:a(-180){-};
\endxy\ar[d]\ar[r]^-{\sim} & L_x^s\ar[d]\ar[r] & G/B\ar[d]^-{\pi} \\  
\bullet\ar[r] & \{y\}\ar[r] & G/P 
}\end{gathered} \end{equation}
The $K$-action induces an isomorphism
$K \qtimes{K_y} L_x^s \mapright{\sim} K\cdot L_x^s$.
Thus,
\[ D_K(K\cdot L_x^s) = D_K(K\qtimes{K_y} L_x^s) = D_{K_y}(L_x^s) = D_{K_y}(\PP^1).\footnote{
This is the analogue of the Lie theoretic principle that `local phenomena is controlled by $SL_2$'.
}
\]

As $|B\backslash G/K|<\infty$, the image of $K_y$ in $\mathrm{Aut}(L^s_x)$ has dimension $\geq 1$.
Identify $\PP^1$ with $\CC\sqcup \{\infty\}$. Modulo conjugation by an element of $\mathrm{Aut}(L^s_x)\simeq PGL_2$, there are four possibilities for the decomposition of $\PP^1$ into $K_y$-orbits:
\begin{description}
\item[Case G] 
$\PP^1$ (the action is transitive);
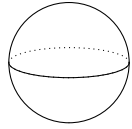
\[
\xy 
(0,0)*\xycircle(8,8){-};
(0,0)*\ellipse(8,2){.};
(0,0)*\ellipse(8,2)__,=:a(-180){-};
\endxy
\]
\item[Case U] 
$\PP^1 = \CC \sqcup \{\infty\}$;
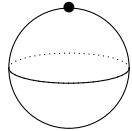
\[
\xy
(0,8)*{\bullet};
(0,0)*\xycircle(8,8){-};
(0,0)*\ellipse(8,2){.};
(0,0)*\ellipse(8,2)__,=:a(-180){-};
\endxy
\]
\item[Case T] 
$\PP^1 = \{0\} \sqcup \CC^* \sqcup \{\infty\}$;
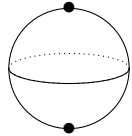
\[
\xy
(0,8)*{\bullet};
(0,-8)*{\bullet};
(0,0)*\xycircle(8,8){-};
(0,0)*\ellipse(8,2){.};
(0,0)*\ellipse(8,2)__,=:a(-180){-};
\endxy
\]
\item[Case N] 
$\PP^1 = \{0,\infty\}\sqcup\CC^*$; both $\{0\}$ and $\{\infty\}$ are fixed points of $K_y^0$.
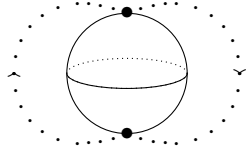
\[
\xy
(0,8)*{\bullet};
(0,-8)*{\bullet};
**\crv{~*=<2mm>{.} (-10,-12) & (-20,0) & (-10, 12)} ?(.5)*\dir{>};
(0,8); (0,-8); **\crv{~*=<2mm>{.} (10,-12) & (20,0) & (10, 12)} ?(.5)*\dir{<};
(0,0)*\xycircle(8,8){-};
(0,0)*\ellipse(8,2){.};
(0,0)*\ellipse(8,2)__,=:a(-180){-};
\endxy
\]
\end{description}
%
We will say that $w$ is of type \textbf{G}, \textbf{U}, \textbf{T} or \textbf{N} relative to $s$ depending on which of these decompositions actually occurs.

Given an irreducible equivariant local system $V_{\tau}$ on a $K$-orbit $j\colon w\hookrightarrow G/B$, set
\[  \cL_{\tau} = j_{!*}V_{\tau}[d_{\tau}], \quad\mbox{where}\quad d_{\tau} = \dim(w). \]
Call $\cL_{\tau}$ \emph{clean} if $\cL_{\tau} \simeq j_!V_{\tau}[d_{\tau}]$. 
Call $\cL_{\tau}$ \emph{cuspidal}
if for each simple reflection $s$, each $v\neq w$ with $s\star v = w$, and each $K$-equivariant local system $V_{\gamma}$ on $v$, the object $\cL_{\tau}$ does \emph{not} occur as a direct summand of $\pH^*(\pi^*\pi_*\cL_{\gamma})$, where $\pi$ is as in \eqref{defpi}.\footnote{
The term `cuspidal' has a very specific meaning in representation theory. It is not clear to me whether the terminology is completely justified in the current geometric setting.}
\begin{lemma}[{\cite[Lemma 7.4.1]{MS}}]\label{clean}
Cuspidals are clean.
\end{lemma}

\begin{proof}
As indicated, this is \cite[Lemma 7.4.1]{MS}. Regardless, the language employed in \cite{MS} is a bit different from ours, so we sketch the argument in order to orient the reader.

Let $j\colon w\hookrightarrow G/B$ be a $K$-orbit, and $V_{\tau}$ a local system on $w$ such that $\cL_{\tau}$ is cuspidal. Write $\overline{w}$ for the closure of $w$.
To demonstrate the assertion we need to show that $(j_*V_{\tau})|_v = 0$ for each orbit $v$ in $\overline{w} - w$.

If $s$ is a simple reflection such that $s\star w = w$ and $P_s\cdot w$ contains an orbit other than $w$, then as $\cL_{\tau}$ is cuspidal, $w$ must be of type \textbf{T} or \textbf{N} relative to $s$. In the language of \cite{MS}, this means that each such $s$ is `of type IIIb or IVb for $w$'. Let $I$ be the set consisting of simple reflections $s$ as above and let $P_I$ be the parabolic subgroup of $G$ containing $B$ and corresponding to $I$. Then in \cite[\S7.2.1]{MS} it is shown that $P_I\cdot w = \overline{w}$.

Now if $v$ is an orbit of codimension $1$ in $\overline{w}$, then
there exists $s\in I$ such that $s\star v = w$.
Inspecting the cases \textbf{T} and \textbf{N} yields the required vanishing in this case.

For arbitrary $v$, proceed by induction on codimension. Let $s\in I$ be such that $s\star v > v$. Let $\pi$ be as in \eqref{defpi}. As $\cL_{\tau}$ is cuspidal, $\pi_*(j_*V_{\tau}) =0$. Furthermore, if $(j_*V_{\tau})|_v\neq 0$, then $\pi_*((j_*V_{\tau})|_v)\neq 0$. All of this follows by inspection of the cases \textbf{G}, \textbf{U}, \textbf{T}, \textbf{N}. Combined these vanishing and non-vanishing statements imply $(j_*V_{\tau})|_v\neq 0$ only if $(j_*V_{\tau})|_{s\star v}\neq 0$. Thus, applying the induction hypothesis yields the result.
\end{proof}
\subsection{Mixed structures}\label{s:mixed}
Given a variety $X$,
write $M(X)$ for the category of $\RR$-mixed Hodge modules on $X$, and $DM(X)$ for its bounded derived category \cite{Sa}. If a linear algebraic group acts on $X$, write $DM_G(X)$ for the corresponding mixed equivariant derived category.
When dealing with mixed as well as ordinary categories, objects in mixed categories will be adorned with an $\vphantom{V}^H$. Omission of the $\vphantom{V}^H$ will denote the classical object underlying the mixed structure.

A mixed Hodge structure is called \emph{Tate} if it is a successive extension of Hodge structures of type $(n,n)$.
A mixed Hodge module $\cA^H\in M(X)$ will be called \emph{$\ast$-pointwise Tate} if, for each point $i\colon \{x\}\hookrightarrow X$, the stalk $H^*(i^*\cA^H)$ is Tate.
Call $\cA^H\in DM(X)$ $\ast$-pointwise Tate if each $\pH^i(\cA^H)$ is so.
An object of $DM_G(X)$ is $\ast$-pointwise Tate if it is so under the forgetful functor $DM_G(X)\to DM(X)$.
\begin{lemma}\label{s:tatepistable}
Let $\pi$ be as in \eqref{defpi}.
Then $\pi^*\pi_*$ preserves the class of $\ast$-pointwise Tate objects.
\end{lemma}
\begin{proof}Use the notation surrounding \eqref{defpi}. Then the assertion reduces to the claim that if $\cA^H\in M_{K_y}(L_x^s)$ is $\ast$-pointwise Tate, then $H^*(L_x^s; \cA^H)$ is Tate.\footnote{
The core of this argument is due to R. MacPherson (in a `parity vanishing' context for Schubert varieties), see \cite[Lemma 3.2.3]{So00}.
}
This is immediate from the possible $K_y$-orbit decompositions \textbf{G}, \textbf{U}, \textbf{T} and \textbf{N}.
\end{proof}

Each irreducible $B\times K$-equivariant local system $V_{\tau}$, on an orbit $w$, underlies a unique (up to isomorphism) polarizable variation of Hodge structure of weight zero. Denote this variation by $V^H_{\tau}$. Taking intermediate extension, we obtain a pure (equivariant) Hodge module $\cL^H_{\tau}$ of weight $d_{\tau} = \dim(w)$, i.e.,
\[ \cL^H_{\tau} = j_{!*}V^H_{\tau}[d_{\tau}], \]
where $j\colon w \hookrightarrow G$ is the inclusion.
\begin{prop}\label{ptwise}$\cL_{\tau}^H$ is $\ast$-pointwise Tate.
\end{prop}
\begin{proof}Work in $G/B$. The statement is true for cuspidals, since they are clean (Lemma \ref{clean}). The general case follows by induction (employing Theorem \ref{order}) and Lemma \ref{s:tatepistable}.
\end{proof}
\begin{prop}\label{contracting}Let $i\colon v\hookrightarrow G$ be the inclusion of a $B\times K$-orbit. Then $i^*\cL^H_{\tau}$ is pure.
\end{prop}
\begin{proof}
Work in $G/K$. According to \cite[\S6.4]{MS} (also see the comments at the end of \S1 in \cite{LV}), each $B$-orbit admits a contracting slice in the sense of \cite[\S2.3.2]{MS}. This implies purity (see \cite[\S2.3.2]{MS} or \cite[Lemma 4.5]{KL} or \cite[Proposition 1]{So89}). 
\end{proof}
%
%
Given an algebraic group $L$ acting on a variety $X$, set
\[ \Ext^{i}_{L}(-,-) = \Hom_{D_L(X)}(-,-[i]).\]
The Hodge modules $\cL^H_{\nu}$ endow each $\Ext^{\bullet}_{B\times K}(\cL_{\tau}, \cL_{\gamma})$ with a Hodge structure. 
\begin{thm}\label{main}$\Ext_{B\times K}^{\bullet}(\cL_{\tau}, \cL_{\gamma})$ is Tate and pure of weight $d_{\gamma}-d_{\tau}$.\footnote{
\textbf{Warning:} the non-equivariant analogue of this result is false!
}
\end{thm}
\begin{proof}
Work in $G/K$. Filtering $\Ext^{\bullet}_B(\cL_{\tau}, \cL_{\gamma})$ by the orbit stratification, one sees that it suffices to argue that $\Ext_B^{\bullet}(i^*\cL_{\tau},i^!\cL_{\gamma}))$ is pure and Tate for each $B$-orbit inclusion $i\colon u\hookrightarrow G/K$. 
Proposition \ref{ptwise} and Proposition \ref{contracting} imply that both $i^*\cL_{\tau}^H$ and $i^!\cL_{\gamma}^H$ are direct sums of (shifted) one dimensional variations of Hodge structure. As $H^*_L$ is pure and Tate, for any linear algebraic group $L$ (cf. \cite[\S9.1]{D}), the assertion is immediate.
\end{proof}
%
\begin{cor}\label{extparity}
$\Ext_{B\times K}^{i}(\cL_{\tau}, \cL_{\gamma}) = 0$ unless $i = d_{\tau} + d_{\gamma}\mod 2$.
\end{cor}
\begin{cor}\label{icparity}
$H^*_{B\times K}(G; \cL_{\tau})$ vanishes in either all even or all odd degrees.
\end{cor}
Some remarks are in order:
\begin{enumerate}
\item Corollary \ref{extparity} should be compared with \cite[Conjecture 4.2.6]{So}; also see \S\ref{s:formality}.
\item Corollary \ref{icparity} is essentially contained in \cite{LV}. Lusztig-Vogan work in the non-equivariant $\ell$-adic setting, but the Hecke algebra computations in \cite{LV} can be used to obtain Corollary \ref{icparity}; also see \S\ref{s:hecke}. Note that Lusztig-Vogan rely on explicit calculations with the Hecke algebra and arguments from representation theory. An argument analogous to \cite{LV}, involving Hecke algebra computations (but no representation theory), can be found in \cite{MS}.
\item Let $P\supseteq B$ be a parabolic subgroup. One should be able to obtain similar results for the analogous $P\times K$-action on $G$ using the technique of \cite{So89}.
\end{enumerate}
%
%
\subsection{Hecke algebra}\label{s:hecke}
Let $L\subseteq G$ be a closed subgroup (we are mainly interested in $L=B$ or $K$).
Let $B\times L$ act on $G$ via $(b,l)\cdot g= bgl^{-1}$. Define a bifunctor 
\[ - \conv - \colon DM_{B\times B}(G) \times DM_{B\times L}(G) \to DM_{B\times L}(G), \]
 called \emph{convolution}, as follows.
For $M\in DM_{B\times B}(G)$, $N\in DM_{B\times L}(G)$, the object $M\boxtimes N$ descends to an object $M\ttimes N \in DM_{B\times L}(G\qtimes{B}G)$. Set $M \conv N = m_!(M\ttimes N)$,
where $m\colon G\qtimes{B} G \to G$ is the map induced by multiplication.
This operation is associative in the evident sense. As $m$ is projective, convolution adds weights and commutes with Verdier duality (up to shift and Tate twist).
%
%
%
%

Taking $L=B$ yields a monoidal structure on $DM_{B\times B}(G)$.
For each $w\in W$, set
\[ \T_w = j_{w!}\const{Bw B} \quad \mbox{and}\quad \C_w = (j_{w!*}\const{Bw B}[\dim(BwB)])[-\dim(BwB)], \]
where $j_w\colon Bw B \hookrightarrow G$ is the inclusion, and $\const{B w B}$ denotes the trivial (weight $0$) variation of Hodge structure on $B w B$. The unit for convolution is $\11 = \T_e$. 
\begin{prop}\label{braid}The $\T_w$ satisfy the braid relations. That is,
if $\ell(vw) = \ell(v) + \ell(w)$, then $\T_v\conv \T_w = \T_{vw}$,
where $\ell\colon W\to \ZZ_{\geq 0}$ is the length function.
\end{prop}
\begin{proof}
Multiplication yields an isomorphism $B v B \qtimes{B} B w B \mapright{\sim} B v  wB$.
 \end{proof}
%
%
\begin{prop}\label{s:convpi}
Let $s\in W$ be a simple reflection, and let $\pi$ be as in \eqref{defpi}. Then, under the equivalence $DM_{B\times K}(G) \mapright{\sim} DM_K(G/B)$, convolution with $\C_s$ is identified with $\pi^*\pi_*$.
\end{prop}
\begin{proof}
Left to the reader (see \cite[Lemma 3.2.1]{So00}).
\end{proof}
Let $\HH_q\subseteq  DM_{B\times B}(G)$ be the triangulated subcategory generated by the $\C_w$, $w\in W$, and Tate twists thereof.
\begin{prop}$\HH_q$ is stable under convolution.
\end{prop}
\begin{proof}
This follows from \cite[Proposition 3.4.6]{So}. Alternatively, note
\[ DM_{B\times B}(G) \mapright{\sim} DM_B(G/B) \mapright{\sim} DM_{G}(G\qtimes{B} G/B) \mapright{\sim}DM_G(G/B\times G/B).\]
This puts us in the setting of the previous sections (with group $G\times G$ and involution $\theta(g_1,g_2)=(g_2,g_1)$). Now use Lemma \ref{s:tatepistable} and Proposition \ref{s:convpi}.
\end{proof}
\begin{cor}Each $\T_w$ is in $\HH_q$.
\end{cor}
\begin{proof}In view of Proposition \ref{braid}, it suffices to prove this for each simple reflection $s$. In this case we have a distinguished triangle 
$\11[-1] \to \T_s \to \C_s \leadsto$
\end{proof}
Let $\MM_q\subseteq DM_{B\times K}(G)$ be the triangulated subcategory generated by the $\cL^H_{\tau}$ and Tate twists thereof.
\begin{thm}\label{mqstable}$\MM_q$ is stable under convolution with objects of $\HH_q$.
\end{thm}
\begin{proof}
Combine Lemma \ref{s:tatepistable} with Proposition \ref{s:convpi}.
\end{proof}
Let $H_q=K_0(\HH_q)$ and $M_q = K_0(\MM_q)$ be the respective Grothendieck groups. These are free $\ZZ[q^{\pm 1}]$-modules via
$q[\cA] = [\cA(-1)]$, where $(-1)$ is the inverse of Tate twist.
Convolution makes $H_q$ a $\ZZ[q^{\pm 1}]$-algebra, and $M_q$ an $H_q$-module.
\begin{cor}\label{positivity}
The coefficients $c^w_{\gamma, \tau}(q)$ in the expansion
\[ (-1)^{d_{\tau}}[\C_w\conv \cL^H_{\tau}] = \sum_{\gamma} (-1)^{d_{\gamma}}c^w_{\gamma,\tau}(q)[\cL^H_{\gamma}], \]
are polynomials in $q^{\pm 1}$ with non-negative coefficients.
\end{cor}
 The algebra $H_q$ is isomorphic to the Iwahori-Hecke algebra associated to $W$. That is, $H_q$ is isomorphic to the $\ZZ[q^{\pm 1}]$ algebra on generators $T_w$, $w\in W$, with relations: $T_vT_w=T_{vw}$ if $\ell(vw) = \ell(v)+\ell(w)$; and $(T_s+1)(T_s-q)=0$ if $\ell(s)=1$.
The isomorphism is given by $T_w\mapsto [\T_w]$.
The convolution product $\C_w\conv \cL^H_{\tau}$, the groups $H^*_{B\times K}(G; \cL_{\tau})$, $\Ext^{\bullet}(\cL_{\tau}, \cL_{\gamma})$, $\Ext^{\bullet}(j_!V_{\tau}, \cL_{\gamma})$, can all be explicitly computed via the module $M_q$:
$\C_w\star \cL^H_{\tau}$ because of Theorem \ref{positivity}; the rest because they are pure and Tate (Theorem \ref{main}) and can
consequently be recovered from their weight polynomials. 
For an explicit description of $M_q$, see \cite{LV} and \cite{MS}.

Corollary \ref{positivity} appears to be new (although it might be possible to deduce it from the results of \cite{LV}). It is a generalization of the well known positivity result for the Kazhdan-Lusztig basis (the classes $[\C_w]$) in the Hecke algebra. 

\subsection{Informal remarks}\label{s:formality}
Let $\cA$ be an abelian category, $\mathrm{Ho}(\cA)$ the homotopy category of chain complexes in $\cA$, and $D(\cA)$ the derived category of $\cA$.
Given a collection of (bounded below) complexes $\{T_i\}$ each of whose components are injectives, set $T=\bigoplus_i T_i$. The complex $\cE = \End^{\bullet}_{\cA}(T)$ has an evident dg-algebra
structure. Let $e_i\in\cE$ denote the idempotent corresponding to projection on $T_i$. The functor $\Hom^{\bullet}(\cE, -)$ yields an equivalence between the full triangulated subcategory of $D(\cA)$ generated by the $T_i$ and the full triangulated subcategory of the dg-derived category $\dgDer{\cE}$ (of right dg $\cE$-modules) generated by the $e_i\cE$.

The dg-algebra $\cE$, and hence $D(\cA)$, becomes significantly more tractable if $\cE$ is \emph{formal}, i.e., quasi-isomorphic to its cohomology $H^*(\cE)$ (viewed as a dg-algebra with trivial differential).
In general, it can be difficult to establish formality. However, there is a criterion due to P. Deligne:
if $\cE$ is endowed with an additional $\ZZ$-grading
$\cE^i = \bigoplus_{j\in \ZZ} \cE^{i,j}$
which is respected by the differential, and each $H^i(\cE)$ is concentrated in degree $i$ (for the additional grading), then $\cE$ is formal.

In the setting of the previous sections, let $\cL = \bigoplus_{\tau}\cL_{\tau}$ be the direct sum of 
the simple objects in $\Perv_{B\times K}(G)$. Let
$\cE = \Ext^{\bullet}_{B\times K}(\cL,\cL)$,
viewed as a dg-algebra with trivial differential.
Assume that the category $\MM_q$ of the previous section is the derived category of an abelian category containing enough injectives. 
Further, assume that the forgetful functor $\MM_q\to D_{B\times K}(G)$ yields a grading (via the weight filtration) in the sense of \cite[\S4]{BGS}.
Then, modulo some finiteness adjectives, Theorem \ref{main} and Deligne's criterion yield
$D_{B\times K}(G) \simeq \dgDer{\cE}$ (this is Soergel's Formality Conjecture). Conjectures 4.2.2, 4.2.3 and 4.2.6 of \cite{So} also follow.

Now $D_{B\times K}(G)$ is not the derived category of an abelian category, but this is not a serious problem for implementing the above argument. However, $\MM_q\to D_{B\times K}(G)$ simply does not yield a grading. The category of Hodge modules is too large. The issue is already visible over a point, since the category of Tate mixed Hodge structures is larger than the category of graded vector spaces. Further, isolating a suitable subcategory of $\MM_q$ seems to be quite difficult (cf. \cite[\S4.5]{BGS}).
W. Soergel has explained to me how combining the arguments of this note with a joint project of his and M. Wendt's on `motivic representation theory' (see \cite{SW}) should allow this basic idea to be carried through (also see footnote \ref{note}). 
This perspective is also explicit in \cite[\S4]{BGS} and \cite[\S G]{B}.

\end{document}